\documentclass [12pt]{amsart}
\usepackage[utf8]{inputenc}
\pdfoutput=1
\usepackage{amsmath,amssymb,amsthm,amsfonts}
\usepackage{mathtools}
\DeclareFontFamily{U}{mathx}{\hyphenchar\font45}
\DeclareFontShape{U}{mathx}{m}{n}{
      <5> <6> <7> <8> <9> <10>
      <10.95> <12> <14.4> <17.28> <20.74> <24.88>
      mathx10
}{}
\DeclareSymbolFont{mathx}{U}{mathx}{m}{n}
\DeclareMathAccent{\widecheck}{0}{mathx}{"71}

\usepackage[english]{babel}
\usepackage{comment}
\usepackage{hyperref}
\usepackage{bbm}
\usepackage{mathrsfs}

\usepackage{tikz,graphicx,color}
\usepackage{tikz-cd}
\usepackage{tikz-3dplot}
\usetikzlibrary{calc}
\usetikzlibrary{arrows}
\usetikzlibrary{shapes}
\usetikzlibrary{patterns}
\usetikzlibrary{positioning}
\usetikzlibrary{arrows.meta}
\usetikzlibrary{decorations.markings}
\usetikzlibrary{knots}

\usepackage{epstopdf}

\usepackage[arrow]{xy}
\usepackage{diagbox}
\usepackage{subfig}
\usepackage{arcs}
\usepackage{xcolor}
\usepackage{xspace}

\usepackage[margin=1.0in]{geometry}

\usepackage{enumitem}
\usepackage[b]{esvect}
\usepackage{changepage}
\usepackage{letltxmacro}
\usepackage{thmtools,etoolbox}

\def\myarabic#1{\normalfont(\roman{#1})}
\newlist{theoremlist}{enumerate}{1}
\setlist[theoremlist]{label=\myarabic{theoremlisti},ref={\myarabic{theoremlisti}},itemindent=0pt,labelindent=0pt,
  leftmargin=*,noitemsep}

\makeatletter
\renewcommand{\p@theoremlisti}{\perh@ps{\thetheorem}}
\protected\def\perh@ps#1#2{\textup{#1#2}}
\newcommand{\itemrefperh@ps}[2]{\textup{#2}}
\newcommand{\itemref}[1]{\begingroup\let\perh@ps\itemrefperh@ps\ref{#1}\endgroup}
\makeatother

\usepackage{nameref,hyperref}
\usepackage[capitalize]{cleveref}

\newtheorem{theorem}{Theorem}[section]

\newtheorem{lemma}[theorem]{Lemma}
\newtheorem{proposition}[theorem]{Proposition}
\newtheorem{corollary}[theorem]{Corollary}
\theoremstyle{definition}
\newtheorem{remark}[theorem]{Remark}
\theoremstyle{definition}
\newtheorem{definition}[theorem]{Definition}
\newtheorem{conjecture}[theorem]{Conjecture}

\theoremstyle{definition}

\theoremstyle{definition}

\crefname{figure}{Figure}{Figures}

\def\figref#1(#2){Figure~\hyperref[#1]{\ref*{#1}(#2)}}

\addtotheorempostheadhook[theorem]{\crefalias{theoremlisti}{theorem}}
\addtotheorempostheadhook[lemma]{\crefalias{theoremlisti}{lemma}}
\addtotheorempostheadhook[proposition]{\crefalias{theoremlisti}{proposition}}
\addtotheorempostheadhook[corollary]{\crefalias{theoremlisti}{corollary}}

% ...

% ...

% ...

\def\<{{\langle}}
\def\>{{\rangle}}

\def\Povtp_#1{\Pi_{#1}^{>0}}
\def\Povtnn_#1{\Pi_{#1}^{\geq0}}

\newcounter{todofigure}

\numberwithin{equation}{section}

\makeatletter
\@namedef{subjclassname@2020}{\textup{2020} Mathematics Subject Classification}
\makeatother

\begin{document}
\numberwithin{equation}{section}

\title{Half-Periodicity of Zamolodchikov Periodic Cluster Algebras}
\author{Ariana Chin}
\address{Department of Mathematics, University of California, Los Angeles, CA 90095, USA}
\email{{\href{mailto:arianagchin@math.ucla.edu}{arianagchin@math.ucla.edu}}}
\thanks{This research did not receive any specific grant from funding agencies in the public, commercial, or not-for-profit sectors.}
\date{\today}



\begin{abstract}
    In 2007, Fomin and Zelevinsky introduced the \textit{bipartite belt}, a sequence of bipartite mutations whose exchange relations form a discrete dynamical system.
    Periodicity of this system is known as \textit{Zamolodchikov periodicity}.
    In our previous work we have classified all Zamolodchikov periodic cluster algebras, but behavior halfway through the period was still unknown.
    This so-called \textit{half-periodicity} was conjectured by Kuniba--Nakanishi--Suzuki
    for $Y$-systems of finite type Cartan matrices, and was proved by Inoue--Iyama--Keller--Kuniba--Nakanishi for tensor products of two simply-laced Dynkin diagrams.
    In this paper, we prove that for any Zamolodchikov periodic cluster algebra, the form at the half-period is a permutation of the cluster variables of order at most two.
\end{abstract}

\maketitle

\setcounter{tocdepth}{1}
\tableofcontents

\section{Introduction}
\subsection{Zamolodchikov periodicity}

Cluster algebras were first introduced in \cite{fomin1} by Fomin and Zelevinsky with applications to dual canonical bases, total positivity in semisimple Lie groups, and Zamolodchikov periodicity for $Y$-systems of finite root systems.
In their fourth paper on cluster algebras, Fomin--Zelevinsky introduced what is known as the \textit{bipartite belt}, a sequence of bipartite mutations defined below whose exchange relations form a discrete dynamical system they called a \textit{generalized $Y$-system} \cite{fomin4}.

An $n\times n$ matrix $B$ is \textit{bipartite} if there exists a coloring $\epsilon: [n]\rightarrow \{\circ, \bullet\}$ of $[n]$ such that for all $i, j$ of the same color, $b_{ij} = b_{ji} = 0$.
It is well known that mutation is a local move, so mutations $\mu_i$ and $\mu_j$ commute if $b_{ij} = b_{ji} = 0$.
Thus in a bipartite matrix, we may define the \textit{bipartite mutations} $\mu_{\circ}$ and  $\mu_{\bullet}$ as the combined mutation at all white vertices and all black vertices, respectively.
A bipartite matrix $B$ is \textit{recurrent} if $\mu_{\circ}\mu_{\bullet}(B) = \mu_{\circ}(-B) = B$. 
The sequence $\mu_{\circ}\mu_{\bullet}\mu_{\circ}\mu_{\bullet}\dots$ is known as the \textit{bipartite belt}.

A bipartite recurrent matrix $B$ is \textit{Zamolodchikov periodic} if the bipartite belt is periodic.
Zamolodchikov periodicity was first observed by Zamolodchikov in his study \cite{zamolodchikov} of thermodynamic Bethe ansatz, initially focusing on the simply-laced Dynkin diagrams $A_n, D_n$, and $E_n$.
In \cite{fomin4}, it was proved that the bipartite belt associated with a generalized Cartan matrix is periodic exactly when the matrix is of finite type.
The conjecture was further proved for all tensor products $\Gamma\otimes \Delta$ of Dynkin diagrams by Keller \cite{keller} and later in a different way by Inoue--Iyama--Keller--Kuniba--Nakanishi \cite{inoue, inoue2}.
In our previous work \cite{chin}, we fully classified all Zamolodchikov periodic cluster algebras.

In particular, the period is a divisor of $2(h + h')$, where $h, h'$ are the Coxeter numbers associated to the bipartite recurrent $B$-matrix.
The behavior after $(h + h')$ mutations is known as \textit{half-periodicity}, and was conjectured in \cite{kuniba} for $Y$-systems and in \cite{inoue2} for $T$-systems of finite type Cartan matrices.
In \cite{inoue}, half-periodicity was proved for tensor products of two simply-laced Dynkin diagrams.
In this paper, we prove half-periodicity for every Zamolodchikov periodic cluster algebra.

\subsection{The main theorem}
The following theorem is the main result of this paper.

\begin{theorem}
\label{mainThm}
    Let $(\Gamma, \Delta)$ be a Zamolodchikov periodic $n\times n$ $B$-matrix with Coxeter numbers $h_{\Gamma}, h_{\Delta}$.
    Let $T(t)$ be the $T$-system associated to $(\Gamma, \Delta)$.
    Then for all $i\in [n]$,
    \begin{align*}
        T_i(t + h_{\Gamma} + h_{\Delta}) = T_{\sigma(i)}(t)
    \end{align*}
    for some automorphism $\sigma\in S_n$ of $(\Gamma, \Delta)$ of order at most two.
    When $h_{\Gamma} + h_{\Delta}$ is even, this permutation is bipartite color preserving.
    Otherwise, it is color reversing.
\end{theorem}

In Section \ref{prelims}, we introduce cluster algebras, maximal green sequences, and bipartite dynamics of cluster algebras, known as \textit{$T$-systems} and \textit{$Y$-systems}.
In Section \ref{sec:maximalGreen}, we prove that for all Zamolodchikov periodic $B$-matrices, the bipartite belt is a maximal green sequence.
In Section \ref{sec:half-periodicity}, we prove the form at the half-period is a permutation of the cluster variables, and observe when it gives the identity permutation.
In Section \ref{sec:coloredMutations}, we observe that in the case of single Dynkin diagrams under a certain global tropical ordering, half-periodicity yields a bijection between almost positive roots and colored tropical mutations, as conjectured in our earlier work \cite{chin}.

\section{Preliminaries}
\label{prelims}

\subsection{Introduction to cluster algebras with coefficients}

An $n\times n$ matrix $B$ is \textit{skew-symmetrizable} if there exists some $c\in \mathbb{R}_{>0}^n$ such that $\forall i, j\in [n]$, $c_ib_{ij} = -c_jb_{ji}$.
In other words, a matrix $B$ is skew-symmetrizable if there exists some nonzero scaling of the rows $c\in \mathbb{R}_{>0}^n$ that makes the matrix skew-symmetric.

\begin{definition}
    Let $I$ be a finite index set.
    \begin{itemize}
        \item[(1)]
            A \textit{semifield} is an abelian multiplicative group $(\mathbb{P}, \oplus, \cdot)$ equipped with addition $\oplus$ which is commutative, associative, and distributive with respect to multiplication in $\mathbb{P}$.
        \item[(2)]
            For an $I$-tuple of variables $y = (y_i)_{i\in I}$, the \textit{universal semifield} $\mathbb{Q}_{sf}(y)$ consists of all rational functions $P(y) / Q(y)$, where $P(y)$ and $Q(y)$ are nonzero polynomials in the $y_i$'s over $\mathbb{Z}_{>0}$, equipped with the usual operations of multiplication and addition.
        \item[(3)]
            For an $I$-tuple of variables $y = (y_i)_{i\in I}$, the \textit{tropical semifield} Trop($y$) is the abelian multiplicative group freely generated by the variables $(y_i)_{i\in I}$, equipped with addition $\oplus$ defined by
            $$\prod_i y_i^{a_i} \oplus \prod_i y_i^{b_i} = \prod_i y_i^{\min(a_i, b_i)}.$$
    \end{itemize}
\end{definition}

\begin{definition}
    A \textit{labeled seed} in field $\mathcal{F}$ is a triple $(x, y, B)$ where
    \begin{itemize}
        \item 
            $B = (b_{ij})$ is an $n\times n$ skew-symmetrizable integer matrix,
        \item 
            $x = (x_1, \dots, x_n)$ is an $n$-tuple of elements in $\mathcal{F}$ forming a \textit{free generating set}; that is, $x_1, \dots, x_n$ are algebraically independent over $\mathbb{QP}$, and $\mathcal{F} = \mathbb{QP}(x_1, \dots, x_n)$,
        \item 
            $y = (y_1, \dots, y_n)$ is an $n$-tuple of algebraically independent elements in $\mathbb{P}$.
    \end{itemize}
\end{definition}
We use the following notation:
\begin{itemize}
    \item 
        $x$ is the (labeled) \textit{cluster} of this seed $(x, B)$,
    \item 
        the elements $x_1, \dots, x_n$ are its \textit{cluster variables},
    \item 
        the elements $y_1, \dots, y_n$ are its \textit{frozen variables} (or \textit{coefficients}),
    \item 
        $B$ is the \textit{exchange matrix}, also known as the $B$-matrix.
\end{itemize}
When the coefficients lie in the tropical semifield Trop($y$), they are called \textit{principal coefficients}.


\begin{definition}
    Let $(x, y, B)$ be a labeled seed. 
    Let $k\in [n]$.
    The \textit{seed mutation} $\mu_k$ in direction $k$ transforms $(x, y, B)$ into the new labeled seed $\mu_k(x, y, B) = (x', y', B')$ defined as follows:
    \begin{itemize}
        \item 
            The cluster $x' = (x_1', \dots, x_n')$ is given by $x_j' = x_j$ for all $j\neq k$, whereas $x_k'$ is determined by the \textit{exchange relation}
            \begin{equation}
            \label{mutationX}
                x_kx_k' = \frac{y_{k}}{1 \oplus y_k}\prod_{b_{ik} > 0} x_i^{|b_{ik}|} + \frac{1}{1\oplus y_k}\prod_{b_{ik} < 0} x_i^{|b_{ik}|}.
            \end{equation}
        \item 
            The coefficients $y' = (y_1', \dots, y_n')$ are given by
            \begin{equation}
            \label{mutationY}
                y_i' = \begin{cases}
                    y_i^{-1} & i = k,\\
                    y_iy_k^{\max(0, b_{ki})}(1\oplus y_k)^{-b_{ki}} & i\neq k.
                \end{cases}
            \end{equation}
        \item  
            The exchange matrix $B' = (b_{ij}')$ is given by
            $$b_{ij}' = \begin{cases}
                -b_{ij} & \text{if }i = k \text{ or } j = k,\\
                b_{ij} + b_{ik}b_{kj} & \text{if } b_{ik}, b_{kj} > 0,\\
                b_{ij} - b_{ik}b_{kj} & \text{if } b_{ik}, b_{kj} < 0,\\
                b_{ij} & \text{otherwise.}
            \end{cases}$$
    \end{itemize}
\end{definition}

\begin{definition}
    Let $(x, y, B)$ be a seed.
    Let $\mathcal{Y}$ be the union of all $\{\frac{y_i}{1\oplus y_i}, \frac{1}{1\oplus y_i}\}$ for every coefficient $y_i$ obtained from any sequence of mutations of this seed.
    Let $\mathcal{X}$ be the set of all cluster variables obtained from any sequence of mutations of this seed.
    The \textit{cluster algebra} $\mathcal{A}$ generated by the seed $(x, y, B)$ is the $\mathbb{ZP}$-subalgebra of the ambient field $\mathcal{F}$ generated by all cluster variables: $\mathcal{A} = \mathbb{ZP}[\mathcal{X}]$.
\end{definition}

In general, $\mathcal{X}$ is an infinite set. When $\mathcal{X}$ is finite, we say the cluster algebra $\mathcal{A}$ is of \textit{finite type}.
The following result introduces an interesting connection between root systems and cluster algebras of finite type.

\begin{theorem}[\cite{fomin}]
\label{fominBijection}
    Let $(B, x, y)$ be a seed of a finite type cluster algebra of type $\Lambda$.
    There is a unique bijection between the almost positive roots of $\Lambda$ and the cluster variables of the cluster algebra generated by $(B, x, y)$, which sends $\alpha\mapsto x[\alpha]$, where
    \[x[\alpha] = \frac{P_{\alpha}(x)}{x^{\alpha}}.\]
    Here, $P_{\alpha}(x)$ is a polynomial over $\mathbb{Z}\mathbb{P}$ with nonzero constant term.
    Moreover, every coefficient of $P_{\alpha}$ is a polynomial in the elements of $\mathcal{Y}$ with positive integer coefficients.
\end{theorem}

\begin{remark}
    In particular, when we specialize at $y_i = 1, \forall i\in [n]$ with principal coefficients, we recover a cluster algebra (without coefficients), and the bijection above becomes $$\alpha\mapsto \frac{P_{\alpha}(x)}{x^{\alpha}},$$
    where $P_{\alpha}(x)$ is a polynomial in $\mathbb{Z}_{\geq 0}[x]$ with nonzero constant term.
\end{remark}

\subsection{Maximal green sequences}

Given a skew-symmetrizable $n\times n$ matrix $B$, the associated \textit{framed matrix} is the $n\times 2n$ extended matrix $\hat{B} = \begin{bmatrix}
    B & I
\end{bmatrix}$.
One can think of this as the exchange matrix with an extra \textit{frozen} vertex $i'$ and directed edge $i\rightarrow i'$ for each mutable vertex $i\in [n]$.

Similarly, the \textit{coframed matrix} is the $n\times 2n$ extended matrix $\check{B} = \begin{bmatrix}
    B & -I
\end{bmatrix}$.

Mutation is not allowed at any frozen vertices.
A mutable vertex $i$ is \textit{green} if $b_{ij}\geq 0$ for all frozen vertices $j\in [n]$.
A mutable vertex $i$ is \textit{red} if $b_{ij}\leq 0$ for all frozen vertices $j\in [n]$.
Notice all mutable vertices are green in the framed matrix $\hat{B}$, and all mutable vertices are red in the coframed matrix $\check{B}$.
In particular, if we start from the framed matrix and apply any sequence of mutations, each mutable vertex will either be green or red.
This result is known as \textit{sign-coherence}, and was proved in full generality in \cite{sign-coherence}.
Every seed of a framed matrix has an associated \textit{$C$-matrix} made up of \textit{$c$-vectors} $c_i$ which encode the adjacency relations between a mutable vertex $i$ and each frozen vertex.
Thus in this terminology, the sign-coherence result states that each $c$-vector $c_i$ either lives in $\mathbb{Z}_{\geq 0}^n$ or in $\mathbb{Z}_{\leq 0}^n$. 

The following result presents a powerful relation between coefficients and $c$-vectors.
\begin{theorem}[Separation Formula \cite{fomin4}]
\label{separation}
    Let $y = (y_i)_{i\in I}$ be the coefficients of the initial seed.
    Let $y'$, $C' = (c_{ij}')$, and $F_j'$ be the coefficients, $C$-matrix, and $F$-polynomials of a seed $\Sigma'$.
    Then,
    \[y_i' = \left(\prod_{j\in I}y_j^{c_{ji}'}\right)\prod_{j\in I}F_j'(y)_{\oplus}^{b_{ji}'}.\]
\end{theorem}
In particular, $F$-polynomials are polynomials with positive integer coefficients and constant term 1.
So if we are working in the tropical semifield Trop($y$) (i.e. with principal coefficients), all $F$-polynomials evaluate to $F_j(y) = 1$, and we are left with $y_i' = y^{c_i'}$, where $c_i'$ is the $i$-th $c$-vector of $C'$.

We say a mutation sequence $k = k_1k_2\dots k_{\ell}$ is \textit{green} if every vertex $k_i$ is green at the time of mutation.
A green sequence $k$ is \textit{maximal} if after applying $\mu_k$, all mutable vertices are red.

The following result states that the bipartite belt is in fact a maximal green sequence for simply-laced Dynkin diagrams $A_n, D_n,$ and $E_n$.

\begin{proposition}[\cite{keller, keller2}]
\label{kellerMaximalGreen}
    Let $Q$ be an alternating ADE quiver with Coxeter number $h$. Let $i_+$ denote the set of vertices that are sources, and $i_-$ denote the set of vertices that are sinks. Then, $i_+i_-$ is a maximal green sequence, and $i_-i_+i_-\dots$ ($h$ factors) is a maximal green sequence.
\end{proposition}

In fact, as presented in \cite{preserving}, given several ``local" maximal green sequences, one may be able to build a ``global" maximal green sequence by \textit{shuffling} together these smaller sequences.

A \textit{shuffle }of two sequences $(a_1, a_2, \dots, a_k)$ and $(b_1, b_2, \dots, b_{\ell})$ is any sequence whose elements are $\{a_1, \dots, a_k\}\cup \{b_1, \dots, b_{\ell}\}$ which preserves the relative order of $(a_1, a_2, \dots, a_k)$ and $(b_1, b_2, \dots, b_{\ell})$.
A \textit{partitioned matrix} $(B, \pi)$ is a partition $\pi = \pi_1 / \pi_2 / \cdots \pi_{\ell}$ of the vertex set of $B$, along with components $B_i$ of $B$ induced by taking the submatrix corresponding to the vertices $\pi_i$.

\begin{definition}
\label{preservingDef}
    Let $(B, \pi)$ be a partitioned matrix.
    A vertex $k\in \pi_i$ is \textit{component preserving} with respect to $\pi$ if one of the following occurs:
    \begin{itemize}
        \item[(1)]
            If $k$ is green in $B$, then $b_{kj} < 0$ implies $j\in \pi_i$.
        \item[(2)]
            If $k$ is red in $B$, then $b_{kj} > 0$ implies $j\in \pi_i$.
    \end{itemize}
\end{definition}

One can think of this as being locally sign-coherent.
A mutation $\mu_k$ at vertex $k$ is component preserving if $k$ is component preserving.
One important property of component preserving mutations with respect to a partition $\pi$ is that they commute with applying $\pi$:
\begin{proposition}
\label{preservingProp}
    Let $(B, \pi)$ be a partitioned skew-symmetric matrix.
    If $k$ is a component preserving vertex, then $\mu_k(B)_i = \mu_k(B_i)$, $\forall i\in [\ell]$.
\end{proposition}

\begin{theorem}[\cite{preserving}]
\label{preservingThm}
    Let $(\hat{B}, \hat{\pi})$ be a partitioned framed skew-symmetric matrix such that for each $\hat{B}_i$, there exists a maximal green sequence $\sigma_i$. Let $\tau$ be a shuffle of $\{\sigma_i\}_{i\in I}$ such that at each mutation step $\mu_k$, $k$ is component preserving with respect to $\pi$. Then, $\tau$ is a maximal green sequence of $\hat{B}$.
\end{theorem}

\subsection{T-systems and Y-systems}

Given a bipartite recurrent $B$-matrix, one can decompose it as a sum of two skew-symmetrizable matrices $B = \tilde{\Gamma} + \tilde{\Delta}$, $\tilde{\Gamma} = (\tilde{\Gamma}_{ij})$ and $\tilde{\Delta} = (\tilde{\Delta}_{ij})$, as follows:
\begin{align*}
    \tilde{\Gamma}_{ij} \coloneqq \begin{cases}
        b_{ij} & \text{if } b_{ij} > 0, \epsilon_i = \circ, \epsilon_j = \bullet,\\
        b_{ij} & \text{if } b_{ij} < 0, \epsilon_i = \bullet, \epsilon_j = \circ,\\
        0 & \text{otherwise.}
    \end{cases}\qquad
    \tilde{\Delta}_{ij} \coloneqq \begin{cases}
        b_{ij} & \text{if } b_{ij} > 0, \epsilon_i = \bullet, \epsilon_j = \circ,\\
        b_{ij} & \text{if } b_{ij} < 0, \epsilon_i = \circ, \epsilon_j = \bullet,\\
        0 & \text{otherwise.}
    \end{cases}
\end{align*}
Throughout the paper, we often use these matrices as unsigned matrices, so we define the matrices $\Gamma = (\Gamma_{ij})$ and $\Delta = (\Delta_{ij})$ by
$$\Gamma_{ij} \coloneqq |\tilde{\Gamma}_{ij}|, \hspace{5mm} \Delta_{ij}\coloneqq |\tilde{\Delta}_{ij}|$$
for all $i,j\in[n]$.
Notice that given a bipartite coloring $\epsilon$ and $(\Gamma, \Delta)$, one can recover $\tilde{\Gamma}, \tilde{\Delta}$.
We often use the abuse of notation $B = (\Gamma, \Delta)$ to refer to the skew-symmetrizable matrix associated to $(\Gamma, \Delta)$.

One can also associate to $B = \tilde{\Gamma} + \tilde{\Delta}$ a \textit{$Y$-system} and a \textit{$T$-system}, two infinite systems of algebraic equations, defined as follows.
Let $x = (x_1, \dots, x_n)$, and let $\mathbb{Q}(x)$ be the field of rational functions in these variables. The \textit{$T$-system} associated with $B$ is a family $T_k(t)$ of elements of $\mathbb{Q}(x)$ satisfying the following relations for all $k\in [n]$ and all $t\in \mathbb{Z}$:
\begin{align*}
    T_k(t+1)T_k(t-1) = \prod_{i} T_i(t)^{\Gamma_{ik}} + \prod_j T_j(t)^{\Delta_{jk}}.
\end{align*}

Similarly, the \textit{$Y$-system} associated with $B$ is a family $Y_k(t)$ of elements of $\mathbb{Q}(x)$ satisfying the following relations for all $k\in [n]$ and all $t\in \mathbb{Z}$:
\begin{align*}
    Y_k(t+1)Y_k(t-1) = \prod_{i} (1 + Y_i(t))^{\Delta_{ik}} + \prod_j (1 + Y_j(t)^{-1})^{-\Gamma_{jk}}
\end{align*}

At each time step, when $\epsilon_k = \circ$, $T_k(t+1)$ and $Y_k(t+1)$ are only dependent on black indices ($\epsilon_i = \bullet$).
Similarly when $\epsilon_k = \bullet$, $T_k(t+1)$ and $Y_k(t+1)$ are only dependent on white indices ($\epsilon_i = \circ$).
Because of this, both systems associated with $B$ split into two independent components.
For our purposes, we consider only one component.
From now on, assume that the two systems are defined only for $k\in [n]$ and $t\in \mathbb{Z}$ such that
\begin{equation}
\label{splitTsystem}
    \epsilon_k = \circ\quad\text{and}\quad t\equiv 0\hspace{-2mm}\pmod{2}\qquad \text{or} \qquad \epsilon_k = \bullet\quad \text{and} \quad t\equiv 1\hspace{-2mm}\pmod{2}.
\end{equation}
Both systems are set to the following initial conditions:
\begin{align*}
    T_k(0) = Y_k(0) = x_k\quad \text{and}\quad T_k(1) = Y_k(1) = x_k, \quad\text{for all } k \text{ satisfying } (\ref{splitTsystem}).
\end{align*}
Notice that the $T$-system relations are exactly the same as the exchange relations (\ref{mutationX}) for cluster mutation in a bipartite recurrent setting, when specialized at $y_i = 1$. In particular, this initialization will result in the same sequence of cluster variables as the mutation sequence $\mu_{\circ}\mu_{\bullet}\mu_{\circ}\mu_{\bullet}\dots$ acting on the initial seed $(B, x)$.
We say the $T$-system or $Y$-system is \textit{periodic} if there exists some positive integer $N$ such that $T_k(t + 2N) = T_k(t)$ for all $k\in[n], t\in\mathbb{Z}$. 
\begin{remark}
In particular, a bipartite recurrent matrix $B$ is Zamolodchikov periodic if its associated $T$-system or $Y$-system is periodic.
The original formulation of Zamolodchikov periodicity was in terms of $Y$-systems.
\end{remark}

\subsection{Tropical $T$-systems}
Given a bipartite recurrent $B$-matrix, we define an analogous system of algebraic equations.

Let $\lambda\in\mathbb{R}^n$ be a labeling of the $n$ indices.
The \textit{tropical $T$-system} $\boldsymbol{t}^{\lambda}$ associated with $B$ is a family $\boldsymbol{t}_k^{\lambda}(t)\in\mathbb{R}$ of real numbers satisfying the following  relations for all $k\in [n]$ and all $t\in\mathbb{Z}$:
\begin{align*}
    \boldsymbol{t}_k^{\lambda}(t + 1) + \boldsymbol{t}_k^{\lambda}(t - 1) = \max\left(\sum_i \Gamma_{ik}\boldsymbol{t}_i^{\lambda}(t), \sum_j \Delta_{jk}\boldsymbol{t}_j^{\lambda}(t)\right).
\end{align*}
Again, we only consider the subsystem defined for $k\in [n]$ and $t\in\mathbb{Z}$ satisfying (\ref{splitTsystem}).
The tropical $T$-system is set to the following initial conditions:
\begin{align*}
    \boldsymbol{t}_k^{\lambda}(0) = \lambda_k\quad\text{and}\quad \boldsymbol{t}_k^{\lambda}(1) &= \lambda_k, \quad \text{for all } k \text{ satisfying } (\ref{splitTsystem}).
\end{align*}
In other words, this system is the \textit{tropicalization} of the $T$-system associated with $B$.
We say the associated tropical $T$-system $\boldsymbol{t}^{\lambda}$ is \textit{periodic} if there exists a positive integer $N$ such that $\boldsymbol{t}_k^{\lambda}(t + 2N) = \boldsymbol{t}_k^{\lambda}(t)$ for all $k\in[n], t\in\mathbb{Z}$.

\begin{remark}
The tropicalization of the $Y$-system can be defined similarly, and its relations exactly model the progression of $c$-vectors, or principal coefficients.
\end{remark}



The following theorem is due to Inoue--Iyama--Keller--Kuniba--Nakanishi.
\begin{theorem}[Periodicity Theorem \cite{inoue}]
\label{periodicity}
    Let $B$ be an arbitrary skew-symmetric matrix, and let $Q$ be the quiver corresponding to $B$.
    Let $(Q(t), x(t), y(t))$ be the seed at $t\in \mathbb{T}_n$ for the cluster algebra $\mathcal{A}(B, x, y)$.
    Let $\hat{y}$ denote the tropical evaluation of $y$ (i.e. the natural mapping from the universal semifield $\mathbb{Q}_{sf}(y)$ to the tropical semifield Trop($y$)).
    Suppose there exists $t\in \mathbb{T}_n$ and $\sigma\in S_n$ such that $\hat{y}_i(t) = \hat{y}_{\sigma(i)}$, $\forall i\in [n]$.
    Then, $x_i(t) = x_{\sigma(i)}$, $y_i(t) = y_{\sigma(i)}$, $\forall i\in [n]$ and $Q(t) = \sigma^{-1}(Q)$.
\end{theorem}


\section{A maximal green sequence}
\label{sec:maximalGreen}

\begin{proposition}
\label{maximalGreen}
    Let $B = (\Gamma, \Delta)$ be a Zamolodchikov periodic $B$-matrix with associated Coxeter numbers $h_{\Gamma}, h_{\Delta}$. 
    Let $\epsilon: [n]\rightarrow \{\circ, \bullet\}$ be a bipartition on $(\Gamma, \Delta)$ such that white vertices are sinks in $\Gamma$ and sources in $\Delta$, and black vertices are sources in $\Gamma$ and sinks in $\Delta$.
    Then, $\mu_{\circ}\mu_{\bullet}\mu_{\circ}\cdots$ ($h_{\Gamma}$ factors) and $\mu_{\bullet}\mu_{\circ}\mu_{\bullet}\cdots$ ($h_{\Delta}$ factors) are maximal green sequences for the framed matrix $\hat{B}$.
\end{proposition}

Before proving the above proposition, we need to first establish a couple of statements.
We say an orientation of a Dynkin diagram is \textit{alternating} if every vertex is either a source or a sink.

As defined in \cite{chin}, a \textit{bicolored automorphism} $f: [n] \rightarrow [n]$ is an automorphism that satisfies the following.
\begin{itemize}
    \item[(i)]
        $\epsilon_{f(i)} = \epsilon_i$ (preserves bipartite coloring of vertices).
    \item[(ii)]
        For all $i_1, i_2$ in the same $f$-orbit $I$, and for all $j\in[n]$, $b_{i_1j}b_{i_2j} \geq 0$ (preserves edge colors).
    \item[(iii)]
        For all $i, j\in [n]$, $b_{f(i)f(j)} = b_{ij}$.
    \item[(iv)]
        For all $i, j$ in the same $f$-orbit $I$, $b_{ij} = b_{ji} = 0$.
\end{itemize}
\begin{lemma}
\label{maxGreenLem}
    Let $B$ be a skew-symmetrizable matrix associated to an alternating Dynkin diagram $\Lambda$ with Coxeter number $h$.
    Let $i_+$ denote the set of vertices that are sources (i.e. $i$ such that $\forall j\in [n],$ $b_{ij} \geq 0$), and let $i_-$ denote the set of vertices that are sinks (i.e. $i$ such that $\forall j\in [n],$ $b_{ij} \leq 0$).
    Then, $i_+i_-$ and $i_-i_+i_-\dots$ ($h$ factors) are maximal green sequences for the framed matrix $\hat{B}$.
\end{lemma}

\begin{proof}
    Follows via a basic folding argument. 
    By Proposition \ref{kellerMaximalGreen}, we can assume $B$ is not skew-symmetric, and folds via some bicolored automorphism $f$ to an ADE type $f(B)$. Recall that the Coxeter number $h$ is invariant of folding.
    If we start from the framed matrix $\hat{B}$, it suffices to show that the sequence of mutations $\mu_{\circ}, \mu_{\bullet}$ commutes with folding. Then, $i_+i_-$ and $i_-i_+i_-\dots$ with $h$ factors will take our matrix $\widehat{f(B)}$ to the coframed matrix $\widecheck{f(B)}$, and thus also take $\hat{B}$ to $\check{B}$.
    We already know from \cite{chin} that folding $B$ commutes with mutation at a single orbit.
    We need only check that after a mutation $\mu_{\circ}$ or $\mu_{\bullet}$, we retain an $f$-admissible matrix $B' = (b_{ij}')$, where $f$ is our bicolored automorphism.

    Recall from \cite{fomin1} that $B$ is $f$-admissible if
    \begin{itemize}
        \item[(1)]
            $b_{f(i)f(j)} = b_{ij}$,
        \item[(2)]
            $f$ sends mutable vertices to mutable vertices, and frozen vertices to frozen vertices,
        \item[(3)]
            if $i\sim i'$, then $b_{ii'} = 0$, and
        \item[(4)]
            if $i\sim i'$, then $b_{ij}b_{i'j}\geq 0$.
    \end{itemize}
    Let $f$ be some bicolored automorphism, which we can extend to an automorphism $\hat{f}$ on $\hat{B}$ (i.e. if $f(i) = j$, then set $\hat{f}(i') = j'$). 
    Then, (1)-(4) already hold for mutable vertices, so it suffices to check the frozen vertices.
    
    (2) and (3) are clear by construction of $\hat{f}$ and the nature of frozen vertices.
    
    Let's consider (1).
    Without loss of generality, we can say $i$ is frozen and $j$ is mutable.
    When $j$ is in the mutation set, $b_{f(i)f(j)}' = -b_{f(i)f(j)} = -b_{ij} = b_{ij}'$.
    The more interesting case is when $j$ is not mutated, but all adjacent vertices (closed under $f$-orbits) are mutated.
    For each mutated orbit $K$, $b_{ij}'$ is obtained by either adding or subtracting $\sum_{k\in K}b_{ik}b_{kj}$, depending on whether $b_{ik}, b_{kj} \geq 0$ or $b_{kj}, b_{kj}\leq 0$ for a generic $k\in K$.
    Notice that $b_{ik}\geq 0\iff b_{f(i)k}\geq 0$ by sign-coherence, and $b_{kj} \geq 0 \iff b_{kf(j)}\geq 0$ as $k,j$ are both mutable.
    In the non-negative case, we get
    \begin{align*}
        b_{ij} + \sum_{k\in K} b_{ik}b_{kj} &= b_{f(i)f(j)} + \sum_{k\in K} b_{f(i)f(k)}b_{f(k)f(j)} = b_{f(i)f(j)} + \sum_{k\in K} b_{f(i)k}b_{kf(j)},
    \end{align*}
    and a similar statement for the non-positive case.
    So, $b_{ij} = b_{f(i)f(j)}$.
    
    It suffices to show (4) for the case that $i' = f(i)$.
    The only interesting case here is when $i$ is frozen and $j$ is mutable, or vice versa.
    When $i$ is frozen, so is $f(i)$, and $b_{ij}b_{f(i)j}\geq 0$ by sign-coherence.
    Otherwise, when $i$ is mutable and $j$ is frozen, we know from (1) that $b_{ij} = b_{f(i)f(j)}$, so
    $$b_{ij}b_{f(i)j} = b_{f(i)f(j)}b_{f(i)j}\geq 0$$
    by sign-coherence once again.
\end{proof}

Below, we state the skew-symmetrizable analogue of Proposition \ref{preservingProp} and Theorem \ref{preservingThm}.
\begin{proposition}
\label{preservingProp2}
    Let $(\hat{B}, \hat{\pi})$ be a partitioned framed skew-symmetrizable matrix. 
    If $k$ is a component preserving vertex, then $\mu_k(\hat{B})_i = \mu_k(\hat{B}_i)$, $\forall i\in [\ell]$.
    Moreover, if each $\hat{B}_i$ has a maximal green sequence $\sigma_i$, let $\tau$ be a shuffle of $\{\sigma_i\}_{i\in I}$ such that at each mutation step, $k$ is component preserving with respect to $\pi$. Then, $\tau$ is a maximal green sequence of $\hat{B}$.
\end{proposition}
\begin{proof}
    The proof is identical to that of the skew-symmetric case, so we refer the reader to \cite{preserving}.
\end{proof}

\begin{proof}[Proof of Proposition \ref{maximalGreen}]
    Let the components in our partition be the connected components of $\Gamma$.
    If we mutate at all white vertices, these are sinks within $\Gamma$ and sources within $\Delta$.
    Initially, we are in the framed quiver, so all mutable vertices are green.
    So, we are in case (1) of the definition of component preserving (\ref{preservingDef}), and we can mutate each component independently.
    By sign-coherence, the global color is the local color of each vertex (red or green).

    Once we have applied $\mu_{\circ}$, all roles have been reversed, and now $\mu_{\bullet}$ mutates at all sinks in $\Gamma$.
    So long as all black vertices are green, this will fall under (1) again and will be component preserving.
    This pattern repeats and as long as our mutation set is green, we are always left in a position where $\mu_{\circ}, \mu_{\bullet}$ is component preserving.
    By Lemma \ref{maxGreenLem}, we know at each step, our mutation set is green, so the mutation sequence $\mu_{\circ}\mu_{\bullet}\mu_{\circ}\cdots$ ($h_{\Gamma}$ factors) is always component preserving.
    By Proposition \ref{preservingProp2}, this sequence is maximal green.

    By a similar argument, $\mu_{\bullet}\mu_{\circ}\mu_{\bullet}\dots$ ($h_{\Delta}$ factors) is also a maximal green sequence.
\end{proof}
\begin{remark}
    The same argument does not work for the maximal green sequence $\mu_{\bullet}\mu_{\circ}$, as mutating at all sources is not component preserving within our chosen partition $\pi$.
\end{remark}
\begin{corollary}
    In the same context as above, $\mu_{\circ}\mu_{\bullet}\mu_{\circ}\cdots$ ($h_{\Gamma}$ factors) and $\mu_{\bullet}\mu_{\circ}\mu_{\bullet}\cdots$ ($h_{\Delta}$ factors) each result in $-C$ being a permutation matrix.
\end{corollary}
\begin{remark}
\label{remark1}
    Beginning from the coframed matrix, applying the mutation $N = h_{\Gamma} + h_{\Delta}$ times results in a \textit{frozen isomorphism}, (i.e. an isomorphism which fixes the frozen vertices).
\end{remark}

\section{Half-periodicity of Zamolodchikov periodic cluster algebras}
\label{sec:half-periodicity}

In this section, we prove half-periodicity for each of the Zamolodchikov periodic $n\times n$ $B$-matrices.
Throughout this section, we use the following notation for every $j\in [n]$.
$$\eta_j \coloneqq \begin{cases}
    0 & \text{if }\epsilon_j = \circ,\\
    1 & \text{if }\epsilon_j = \bullet.
\end{cases}$$
We start by analyzing the potential frozen isomorphisms that appear, as described in Remark \ref{remark1}.
Examples of these isomorphisms for types $A_{5}\ast D_4$ and $F_4\otimes A_2$ are depicted in Figures \ref{fig:AD} and \ref{fig:BA}.
We use red edges to depict the $\Gamma$-components, and blue edges to depict the $\Delta$-components.

\begin{figure}
\scalebox{0.6}{
\begin{tikzpicture}
\draw (-1.00, 0.00) node [anchor=north west][inner sep=0.75pt] {\Large{$x_1$}};
\draw (-1.00, 2.00) node [anchor=north west][inner sep=0.75pt] {\Large{$x_2$}};
\draw (1.50, 3.75) node [anchor=north west][inner sep=0.75pt] {\Large{$x_3$}};
\draw (2.50, 1.00) node [anchor=north west][inner sep=0.75pt] {\Large{$x_4$}};
\draw (2.50, -1.00) node [anchor=north west][inner sep=0.75pt] {\Large{$x_5$}};
\draw (6.50, -0.50) node [anchor=north west][inner sep=0.75pt] {\Large{$x_6$}};
\draw (6.50, 1.50) node [anchor=north west][inner sep=0.75pt] {\Large{$x_7$}};
\draw (5.50, 3.75) node [anchor=north west][inner sep=0.75pt] {\Large{$x_8$}};
\draw (8.50, 2.50) node [anchor=north west][inner sep=0.75pt] {\Large{$x_9$}};

\draw (14 + -1.00, 0.00) node [anchor=north west][inner sep=0.75pt] {\Large{$x_1$}};
\draw (14 + -1.00, 2.00) node [anchor=north west][inner sep=0.75pt] {\Large{$x_2$}};
\draw (14 + 1.50, 3.75) node [anchor=north west][inner sep=0.75pt] {\Large{$x_3$}};
\draw (14 + 2.50, 1.00) node [anchor=north west][inner sep=0.75pt] {\Large{$x_4$}};
\draw (14 + 2.50, -1.00) node [anchor=north west][inner sep=0.75pt] {\Large{$x_5$}};
\draw (14 + 6.50, -0.50) node [anchor=north west][inner sep=0.75pt] {\Large{$x_6$}};
\draw (14 + 6.50, 1.50) node [anchor=north west][inner sep=0.75pt] {\Large{$x_7$}};
\draw (14 + 5.50, 3.75) node [anchor=north west][inner sep=0.75pt] {\Large{$x_9$}};
\draw (14 + 8.50, 2.50) node [anchor=north west][inner sep=0.75pt] {\Large{$x_8$}};

\foreach \k in {0, 1}
{
    \coordinate (\k v0x0) at (14 * \k + 0.00,0.00);
    \coordinate (\k v0x1) at (14 * \k + 0.00,2.00);
    \coordinate (\k v0x2) at (14 * \k + 2.00,3.00);
    \coordinate (\k v0x3) at (14 * \k + 3.50,1.00);
    \coordinate (\k v0x4) at (14 * \k + 3.50,-1.00);
    \coordinate (\k v1x0) at (14 * \k + 6.00,-0.50);
    \coordinate (\k v1x1) at (14 * \k + 6.00,1.50);
    \coordinate (\k v1x2) at (14 * \k + 5.00,3.50);
    \coordinate (\k v1x3) at (14 * \k + 8.00,2.50);
    \foreach \i in {0, 1, 2, 3}
    {
        \pgfmathtruncatemacro{\j}{\i + 1};
        \draw[color=red,line width=0.75mm] (\k v0x\i) to[] (\k v0x\j);
    }
    \draw[color=red,line width=0.75mm] (\k v1x0) to[] (\k v1x1);
    \draw[color=red,line width=0.75mm] (\k v1x1) to[] (\k v1x2);
    \draw[color=red,line width=0.75mm] (\k v1x1) to[] (\k v1x3);
    \foreach \i in {0, 1}
    {
        \pgfmathtruncatemacro{\j}{4 - \i};
        \draw[color=blue,line width=0.75mm] (\k v0x\i) to[] (\k v1x\i);
        \draw[color=blue,line width=0.75mm] (\k v0x\j) to[] (\k v1x\i);
    }
    \draw[color=blue,line width=0.75mm] (\k v0x2) to[] (\k v1x2);
    \draw[color=blue,line width=0.75mm] (\k v0x2) to[] (\k v1x3);
    
    \draw[fill=black!75!white] (\k v0x0.center) circle (0.2);
    \draw[fill=white] (\k v0x1.center) circle (0.2);
    \draw[fill=black!75!white] (\k v0x2.center) circle (0.2);
    \draw[fill=white] (\k v0x3.center) circle (0.2);
    \draw[fill=black!75!white] (\k v0x4.center) circle (0.2);
    \draw[fill=white] (\k v1x0.center) circle (0.2);
    \draw[fill=black!75!white] (\k v1x1.center) circle (0.2);
    \draw[fill=white] (\k v1x2.center) circle (0.2);
    \draw[fill=white] (\k v1x3.center) circle (0.2);
}
\draw[color=black, line width=0.75mm] (10.00 + 0.56, 1.50 + -0.22) to[] (10.00 + 0.56, 1.50 + 0.22) ;
\draw[color=black, line width=0.75mm] (10.00 + 1.44, 1.50 + 0.0) to[] (10.00 + 0.56, 1.50 + 0.00) ;
\draw[color=black, line width=0.75mm] (10.00 + 1.44, 1.50 + 0.0) to[] (10.00 + 1.22, 1.50 + 0.22) ;
\draw[color=black, line width=0.75mm] (10.00 + 1.44, 1.50 + 0.0) to[] (10.00 + 1.22, 1.50 + -0.22) ;
\end{tikzpicture}}
\caption{\label{fig:AD} The color preserving permutation of the cluster variables for $A_5\ast D_4$. On the left is the initial cluster at $t = 0$. On the right is the permuted cluster variables at the half-period $t = 10$.}
\end{figure}

\begin{figure}
\scalebox{0.6}{
\begin{tikzpicture}
\draw (-1.00, 0.00) node [inner sep=0.75pt] {\Large{$x_1$}};
\draw (-1.00, 2.00) node [inner sep=0.75pt] {\Large{$x_2$}};
\draw (-1.00, 4.00) node [inner sep=0.75pt] {\Large{$x_3$}};
\draw (-1.00, 6.00) node [inner sep=0.75pt] {\Large{$x_4$}};
\draw (5.00, 0.00) node [inner sep=0.75pt] {\Large{$x_5$}};
\draw (5.00, 2.00) node [inner sep=0.75pt] {\Large{$x_6$}};
\draw (5.00, 4.00) node [inner sep=0.75pt] {\Large{$x_7$}};
\draw (5.00, 6.00) node [inner sep=0.75pt] {\Large{$x_8$}};

\draw (17.00, 0.00) node [inner sep=0.75pt] {\Large{$x_1$}};
\draw (17.00, 2.00) node [inner sep=0.75pt] {\Large{$x_2$}};
\draw (17.00, 4.00) node [inner sep=0.75pt] {\Large{$x_3$}};
\draw (17.00, 6.00) node [inner sep=0.75pt] {\Large{$x_4$}};
\draw (11.00, 0.00) node [inner sep=0.75pt] {\Large{$x_5$}};
\draw (11.00, 2.00) node [inner sep=0.75pt] {\Large{$x_6$}};
\draw (11.00, 4.00) node [inner sep=0.75pt] {\Large{$x_7$}};
\draw (11.00, 6.00) node [inner sep=0.75pt] {\Large{$x_8$}};

\foreach \k in {0, 1}
{
\coordinate (\k v0x0) at (12 * \k + 0.00,0.00);
\coordinate (\k v0x1) at (12 * \k + 0.00,2.00);
\coordinate (\k v0x2) at (12 * \k + 0.00,4.00);
\coordinate (\k v0x3) at (12 * \k + 0.00,6.00);

\coordinate (\k v1x0) at (12 * \k + 4.00,0);
\coordinate (\k v1x1) at (12 * \k + 4.00,2);
\coordinate (\k v1x2) at (12 * \k + 4.00,4);
\coordinate (\k v1x3) at (12 * \k + 4.00,6);

\draw[color=red, line width=0.75mm] (12 * \k + 0.0, 3.11) to[] (12 * \k + -0.22, 2.89) ;
\draw[color=red, line width=0.75mm] (12 * \k + 0.0, 3.11) to[] (12 * \k + 0.22, 2.89) ;
\draw[color=red, line width=0.75mm] (12 * \k + -0.07, 2.0) to[] (12 * \k + -0.07, 4.0) ;
\draw[color=red, line width=0.75mm] (12 * \k + 0.07, 2.0) to[] (12 * \k + 0.07, 4.0) ;
\draw[color=red, line width=0.75mm] (12 * \k + 4.0, 3.11) to[] (12 * \k + 3.78, 2.89) ;
\draw[color=red, line width=0.75mm] (12 * \k + 4.0, 3.11) to[] (12 * \k + 4.22, 2.89) ;
\draw[color=red, line width=0.75mm] (12 * \k + 3.93, 2.0) to[] (12 * \k + 3.93, 4.0) ;
\draw[color=red, line width=0.75mm] (12 * \k + 4.07, 2.0) to[] (12 * \k + 4.07, 4.0) ;

\foreach \i in {0, 1, 2, 3}
{
    \draw[color=blue,line width=0.75mm] (\k v0x\i) to[] (\k v1x\i);
}
\foreach \i in {0, 1}
{
    \draw[color=red,line width=0.75mm] (\k v\i x0) to[] (\k v\i x1);
    \draw[color=red,line width=0.75mm] (\k v\i x2) to[] (\k v\i x3);
}
\draw[fill=white] (\k v0x1.center) circle (0.2);
\draw[fill=white] (\k v0x3.center) circle (0.2);
\draw[fill=white] (\k v1x0.center) circle (0.2);
\draw[fill=white] (\k v1x2.center) circle (0.2);
\draw[fill=black!75!white] (\k v0x0.center) circle (0.2);
\draw[fill=black!75!white] (\k v0x2.center) circle (0.2);
\draw[fill=black!75!white] (\k v1x1.center) circle (0.2);
\draw[fill=black!75!white] (\k v1x3.center) circle (0.2);
}
\draw[color=black, line width=0.75mm] (7.00 + 0.56, 3.00 + -0.22) to[] (7.00 + 0.56, 3.00 + 0.22) ;
\draw[color=black, line width=0.75mm] (7.00 + 1.44, 3.00 + 0.0) to[] (7.00 + 0.56, 3.00 + 0.00) ;
\draw[color=black, line width=0.75mm] (7.00 + 1.44, 3.00 + 0.0) to[] (7.00 + 1.22, 3.00 + 0.22) ;
\draw[color=black, line width=0.75mm] (7.00 + 1.44, 3.00 + 0.0) to[] (7.00 + 1.22, 3.00 + -0.22) ;
\end{tikzpicture}}
\caption{\label{fig:BA} The color reversing permutation of the cluster variables for $F_4\otimes A_2$. On the left is the initial cluster at $t = 0$. On the right is the permuted cluster variables at the half-period $t = 15$.}
\end{figure}
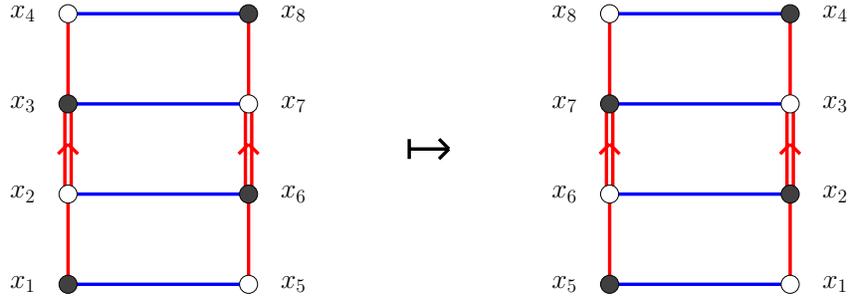

\begin{lemma}
\label{isomLem}
    A frozen isomorphism of a Zamolodchikov periodic $B$-matrix as described above in Remark \ref{remark1} must preserve the bipartite coloring when $N = h_{\Gamma} + h_{\Delta}$ is even, and reverse the coloring when $N$ is odd.
\end{lemma}

\begin{proof}
    Recall that $B$ is recurrent, so every application of $\mu_{\circ}, \mu_{\bullet}$ switches the source-sink structure of $B$.
    Thus for odd $N$, we are looking at isomorphisms $B\xrightarrow{\sim} -B$, and for even $N$, we are looking at isomorphisms $B\xrightarrow{\sim} B$.
    
    First, we consider ADE bigraphs, which were classified in \cite{stembridge}. 
    Let $N$ be odd.
    Notice that $A_{2n}$ is the only Dynkin diagram with an odd Coxeter number $h$.
    So the only case where we get odd $N$ is if we have a tensor product of the form $A_{2n}\otimes \Delta$, where $\Delta\not\cong A_{2n}$.
    If the isomorphism is color preserving, then to preserve the source-sink structure of $B$ at time $t = N$, it must necessarily reverse the edge colors, mapping $\Gamma\mapsto \Delta$.
    However, $\Delta\not\cong A_{2n}$, so this cannot be an isomorphism.
    Thus, our isomorphism must be color reversing.

    Now when $N$ is even, any isomorphism that is color reversing must necessarily reverse the colors of the edges as well as the vertices to preserve the source-sink structure of $B$.
    So if $\Delta\not\cong \Gamma$, then the claim already holds.
    This leaves tensor products $\Gamma\otimes \Gamma$, twists $\Gamma\times \Gamma$, $(AD^{n-1})_n, (A^{n-1}D)_n, E_6\equiv E_6\ast E_6\equiv E_6, D_6\ast D_6\equiv D_6, E_8\ast E_8\equiv E_8\equiv E_8, D_5\boxtimes A_7$, and $E_7\boxtimes D_{10}$.
    
    In each of these cases (except tensor products), there is a unique automorphism that switches $\Gamma$ and $\Delta$, but it is vertex color preserving, not reversing.
    For tensor products, if the number of $\circ$ vertices is distinct from the number of $\bullet$ vertices (as is the case for $A_{2n+1}, D_n$, and $E_7$), there cannot exist a color reversing automorphism.
    For $E_6\otimes E_6$ and $E_8\otimes E_8$, there is a unique automorphism that switches $\Gamma$ and $\Delta$, but it is vertex color preserving, not reversing.

    The only remaining case is $A_{2n}\otimes A_{2n}$. There does indeed exist a color reversing automorphism, but due to \cite{pasha}, we understand the behavior of $A_n\otimes A_m$ completely.
    At time $t = N$, the frozen isomorphism is an $180^{\circ}$ rotation, which is color preserving.

    So, all ADE bigraphs satisfy the claim.

    Now, we explore what happens when we apply folding and transpose to the ADE bigraphs.
    Recall that Coxeter numbers are invariant of folding and taking transpose.
    All isomorphisms of folded biagrams are the same as their unfolded versions, and all foldings we do are along bicolored automorphisms, which are color preserving.
    In addition, isomorphisms for $B$ are the same as isomorphisms for its Langland dual $-B^{\top}$.
    By \cite[Proposition~3.22]{chin}, every Zamolodchikov periodic $B$-matrix is obtainable from an ADE bigraph through foldings and transpose.
    So the claim is true for all Zamolodchikov periodic $B$-matrices.
\end{proof}

The following proposition states Theorem \ref{mainThm} for simply-laced types.
\begin{proposition}
\label{ADE}
    Let $(\Gamma, \Delta)$ be an ADE bigraph with Coxeter numbers $h_{\Gamma}, h_{\Delta}$.
    Let $T(t)$ be the $T$-system associated to $(\Gamma, \Delta)$.
    Then for all $i\in [n]$,
    \begin{align*}
        T_i(t + h_{\Gamma} + h_{\Delta}) = T_{\sigma(i)}(t)
    \end{align*}
    for some automorphism $\sigma\in S_n$ of $(\Gamma, \Delta)$ of order at most two.
    When $h_{\Gamma} + h_{\Delta}$ is even, this permutation is bipartite color preserving.
    Otherwise, it is color reversing.
\end{proposition}
\begin{proof}
    Let the initial seed be our coframed matrix $\check{B}$, and let $\hat{y}_i$ denote the tropical evaluation of the coefficient $y_i$ (in the tropical semifield Trop($y$)).
    As noted in Remark \ref{remark1}, applying $\mu_{\circ}\mu_{\bullet}\cdots$ ($h_{\Gamma} + h_{\Delta}$ factors) gives us a frozen isomorphism, or a permutation $\sigma$ of the $C$-matrix.
    Moreover by the separation formula (\ref{separation}), $c$-vectors are exactly the exponent vectors for principal coefficients, the tropical evaluation of coefficients, so we have $\hat{y}_i(t) = \hat{y}_{\sigma(i)}$, $\forall i\in [n]$.
    Thus by the periodicity theorem (\ref{periodicity}), the cluster variables also exhibit the same permutation $\sigma$ of variables at time $N = h_{\Gamma} + h_{\Delta}$.

    Recall that each application of $\mu_{\circ}, \mu_{\bullet}$ switches the source-sink structure of $(\Gamma, \Delta)$.
    
    By Lemma \ref{isomLem}, we know when $N$ is even, the isomorphism is color preserving.
    This means that at time $t = N$, we are left with exactly the same matrix $\check{B}$, and applying the mutation for the second half of the period will apply the same isomorphism again.
    When $N$ is odd, the isomorphism is color reversing. 
    So at time $t = N$, we are left in a reversed position, where all sources and sinks have switched roles and colors.
    Since $N$ is odd, our mutation sequence moving into the second half period is also reversed, so applying the mutation sequence for the second half will apply the same isomorphism again.
    Thus, the second half of the period behaves identically to the first half, with the clusters permuted at each step.
    Since $B$ is Zamolodchikov periodic with period $2N$, this frozen isomorphism must be of order at most two.
\end{proof}

The next three lemmas help prove that half-periodicity commutes with taking transpose.

\begin{lemma}
\label{tropicalT}
    Let $B$ be a bipartite recurrent $B$-matrix.
    Let $\sigma\in S_n$ be a bipartite color preserving automorphism of $B$, and let $N\in \mathbb{Z}_{>0}$ be an even integer. 
    The following are equivalent:
    \begin{itemize}
        \item[(1)]
            $\boldsymbol{t}^{\delta_j}_i(t + N) = \boldsymbol{t}^{\delta_j}_{\sigma(i)}(t)$ for all $i, j\in [n], t\in \mathbb{Z}$.
        \item[(2)]
            $\boldsymbol{t}^{\lambda}_i(t + N) = \boldsymbol{t}^{\lambda}_{\sigma(i)}(t)$ for all $i\in [n], t\in \mathbb{Z}$, for any initial labeling $\lambda\in \mathbb{R}^n$.
        \item[(3)]
            There exists a labeling $c\in\mathbb{R}_{>0}^n$ such that $T_i(t + N) = c_{\sigma(i)}T_{\sigma(i)}(t)$ for all $i\in [n], t\in\mathbb{Z}$.
    \end{itemize}
\end{lemma}
\begin{proof}

    (2)$\implies$(1) is trivial, and (3)$\implies$(2) follows directly from \cite[Lemma~4.2]{chin}.

    To show (1)$\implies$(3), it is enough to consider $t = \eta_j$, because if $T_i(\eta_i + N) = c_{\sigma(i)}T_{\sigma(i)}(\eta_{\sigma(i)})$ for all $i\in [n]$, then we can get the case of arbitrary $t$ via the substitution $x_j \coloneqq T_i(t + \eta_i)$ for all $i\in [n]$.
    By \cite[Lemma~4.3]{chin}, (1) implies that for every $i, j\in [n]$,
    \begin{align*}
        \deg_{\max}(j, T_i(\eta_i + N)) &= \deg_{\max}(j, T_{\sigma(i)}(\eta_{\sigma(i)})) \quad \text{and}\\
        \deg_{\min}(j, T_i(\eta_i + N)) &= \deg_{\min}(j, T_{\sigma(i)}(\eta_{\sigma(i)})).
    \end{align*}
    Since $T_{\sigma(i)}(\eta_{\sigma(i)}) = x_{\sigma(i)}$, $\deg_{\max}(j, T_{\sigma(i)}(\eta_{\sigma(i)})) = \deg_{\min}(j, T_{\sigma(i)}(\eta_{\sigma(i)})) = \delta_{\sigma(i)}(j)$ and therefore $T_i(\eta_i + N)$ differs from $T_{\sigma(i)}(\eta_{\sigma(i)})$ by a scalar multiple $c_{\sigma(i)}$. 
    Since they are both Newton-positive, $c_{\sigma(i)}$ is necessarily positive, so (3) follows.
\end{proof}

\begin{remark}
    The above proof is identical to the proof of \cite[Proposition~4.4]{chin}, but is a slightly stronger statement.
\end{remark}

\begin{lemma}
\label{scalarsCommute}
    Let $B$ be a bipartite recurrent skew-symmetrizable matrix. 
    Let $c\in \mathbb{R}_{>0}^n$ be the vector of scalars such that $c_ib_{ij} = c_jb_{ji}$ $\forall i, j\in [n]$. 
    Let $\sigma\in S_n$ be an automorphism of $B$. Then, $c_{\sigma(i)} = c_i$, $\forall i\in [n]$.
\end{lemma}

\begin{proof}
    After applying the permutation, we get
    $c_{\sigma(i)}b_{\sigma(i)\sigma(j)} = -c_{\sigma(j)}b_{\sigma(j)\sigma(i)}.$
    
    Since $b_{\sigma(i)\sigma(j)} = b_{ij}$, this new permuted vector still works to symmetrize $B$:
    \begin{align*}
        c_{\sigma(i)}b_{ij} = c_{\sigma(i)}b_{\sigma(i)\sigma(j)} = -c_{\sigma(j)}b_{\sigma(j)\sigma(i)} = -c_{\sigma(j)}b_{ji}, \forall i, j\in [n].
    \end{align*}
    As $B$ is strongly connected, $\begin{bmatrix}
        c_{\sigma(1)}\\
        c_{\sigma(2)}\\
        \vdots\\
        c_{\sigma(n)}
    \end{bmatrix}$ is a scalar multiple of $\begin{bmatrix}
        c_{1}\\
        c_{2}\\
        \vdots\\
        c_{n}
    \end{bmatrix}$.
    Thus, $c_{\sigma(i)} = c_i$, $\forall i\in [n]$.
\end{proof}

\begin{lemma}
\label{transposeCommutes}
    Let $\sigma\in S_n$ be a bipartite color preserving automorphism of $B$, and let $N\in \mathbb{Z}_{>0}$ be an even integer. 
    For any initial labeling $\lambda\in\mathbb{R}^n$, let the tropical $T$-system of $B$ satisfy
    $$\boldsymbol{t}^{\lambda}_i(t + N) = \boldsymbol{t}^{\lambda}_{\sigma(i)}(t), \text{ for all } i\in [n], t\in \mathbb{Z}.$$
    Then, the tropical $T$-system of the Langland dual $-B^{\top}$ satisfies the same relations.
\end{lemma}

\begin{proof}
    Let $c\in \mathbb{R}_{>0}^n$ be the vector of scalars such that $c_ib_{ij} = c_jb_{ji}$ $\forall i, j\in [n]$.
    Let $\boldsymbol{t}'$ denote the tropical $T$-system of $-B^{\top}$.
    It is enough to show $\boldsymbol{t}'^{\lambda}_i(\eta_i + N) = \boldsymbol{t}'^{\lambda}_{\sigma(i)}(\eta_{\sigma(i)})$.
    From \cite[Lemma~4.6]{chin} and Lemma \ref{scalarsCommute}, moving between the tropical $T$-systems of the Langland duals $B$ and $-B^{\top}$ looks like
    \begin{align*}
        \boldsymbol{t'}_i^{\lambda}(\eta_i + N) &= \frac{1}{c_i} \boldsymbol{t}_i^{\tilde{\lambda}}(\eta_i + N) = \frac{1}{c_i} \boldsymbol{t}_{\sigma(i)}^{\tilde{\lambda}}(\eta_{\sigma(i)}) = \frac{1}{c_{\sigma(i)}} \boldsymbol{t}_{\sigma(i)}^{\tilde{\lambda}}(\eta_{\sigma(i)})
        =  \boldsymbol{t'}_{\sigma(i)}^{\lambda}(\eta_{\sigma(i)}).
    \end{align*}
    Here, $\tilde{\lambda}\in \mathbb{R}^n$ is the labeling defined by $\tilde{\lambda}_i = c_i\lambda_i$ for each $i\in [n]$.
\end{proof}

Now, we are ready to prove half-periodicity.
\begin{proof}[Proof of Theorem \ref{mainThm}]
    Let $N = h_{\Gamma} + h_{\Delta}.$
    From \cite[Proposition~3.22]{chin}, any Zamolodchikov periodic $B$-matrix can be obtained from an ADE bigraph via folding and taking transpose.
    By Proposition \ref{ADE}, it suffices to show half-periodicity is preserved under folding and taking transpose.

    It is clear that half-periodicity is preserved under folding.
    A cluster variable $x_I'$ of $(f(\Gamma), f(\Delta))$ for an $f$-orbit $I$ is exactly equal to the cluster variable $x_i'$ of $(\Gamma, \Delta)$ for any $i\in I$, under the initial specialization $x_{f(i)} = x_i$ for all $i\in I$.

    We also claim that half-periodicity for even $N$ is preserved under taking transpose.
    By Lemma \ref{tropicalT}, if we start with a matrix $B$ such that $T_i(t + N) = c_{\sigma(i)}T_{\sigma(i)}(t)$ for all $i\in [n]$, then the tropical $T$-system satisfies
    $\boldsymbol{t}^{\lambda}_i(t + N) = \boldsymbol{t}^{\lambda}_{\sigma(i)}(t)$ for all $i\in [n], t\in \mathbb{Z}$, for every initial labeling $\lambda\in \mathbb{R}^n$.
    By Lemma \ref{transposeCommutes}, the same relations hold for the tropical $T$-system $\boldsymbol{t}'$ associated to the Langland dual $-B^{\top}.$
    Applying Lemma \ref{tropicalT} once more tells us the $T$-system $T'$ associated with $-B^{\top}$ will satisfy
    $$T_i'(t + N) = c_{\sigma(i)}T_{\sigma(i)}'(t), \forall i\in [n].$$
    
    Since the $T$-system is periodic with period $2N$, we have $\sigma^2 = 1$ and
    \[T_i(t + 2N) = c_{\sigma(i)}T_{\sigma(i)}(t + N) = c_ic_{\sigma(i)}T_i(t) = T_i(t).\]
    Thus, $c_ic_{\sigma(i)} = 1$ for every $i\in [n]$.
    By the Laurent phenomenon (\cite{fomin1}), all cluster variables are Laurent polynomials in the initial cluster with integral coefficients, so in fact $c_i = 1$ for all $i\in [n]$.

    In the construction in \cite{chin}, transpose is never applied to a matrix $B$ with odd $N$ to obtain one of the non-ADE families.
    So, these operations generate all Zamolodchikov periodic $B$-matrices.
\end{proof}

In fact, many of these families have no nontrivial automorphisms of order two.
In such cases, half-periodicity yields the identity permutation at time $t = h_{\Gamma} + h_{\Delta}$.
The families which do so are listed below.

\begin{corollary}
    Let $B = (\Gamma, \Delta)$ be one of the following families of rank $n$:
    \begin{itemize}
        \item[(1)] 
            tensor products $\Gamma\otimes \Delta$, for $\Gamma, \Delta\in \{B_n, C_n, E_7, E_8, F_4, G_2\}$,
        \item[(2)] 
            double bindings $B_n\bowtie C_n$, $B_3\bowtie_{1, 2} G_2$, $C_3\bowtie_{1, 2} G_2$, $B_4\boxtimes C_4$,
        \item[(3)]
            $D_6\ast D_6$, $D_6\ast D_6\equiv D_6$, $D_6\ast D_6\equiv D_6\equiv D_6$,
        \item[(4)]
            $E_8\ast E_8$, $E_8\ast E_8\equiv E_8$, $E_8\ast E_8\equiv E_8 \equiv E_8$,
        \item[(5)]
            $B_n\ast C_n$, $(BC^{m-1})_n$, $(B^{m-1}C)_n$, $B_n\equiv B_n\ast C_n\equiv C_n$,
        \item[(6)]
            $F_4\ast F_4\equiv \dots \equiv F_4$,
        \item[(7)]
            $E_7\boxtimes D_{10}$.
    \end{itemize}
    Then, the $T$-system associated to $B$ satisfies $T_i(t + h_{\Gamma} + h_{\Delta}) = T_i(t)$, for all $i\in [n]$.
\end{corollary}
\begin{proof}
    It is a small amount of casework to check that each of these families have no nontrivial automorphisms of order two.
\end{proof}

\section{Colored mutations}
\label{sec:coloredMutations}
Let $B$ be a Zamolodchikov periodic $B$-matrix, and let $\boldsymbol{t}_i^{\lambda}$ be its corresponding tropical $T$-system with initial labeling $\lambda\in \mathbb{R}^n$.
Recall the form for tropical mutation at index $k\in [n]$:
\begin{align*}
    \boldsymbol{t}_k^{\lambda}(t + 1) + \boldsymbol{t}_k^{\lambda}(t - 1) = \max\left(\sum_i\Gamma_{ik}\boldsymbol{t}_i^{\lambda}(t), \sum_j \Delta_{jk}\boldsymbol{t}_j^{\lambda}(t)\right).
\end{align*}
For a generic initial labeling $\lambda\in\mathbb{R}^n$, we can color this mutation either red or blue, depending on the evaluation of the max function. 
If the maximum is $\sum_i\Gamma_{ik}\boldsymbol{t}_i^{\lambda}(t)$, then we call this a \textit{$\Gamma$-mutation}, and color it red.
Otherwise if the maximum is $\sum_j\Delta_{jk}\boldsymbol{t}_j^{\lambda}(t)$, then we call this a \textit{$\Delta$-mutation}, and color it blue.

Recall that the number of roots in a root system is $h\cdot r$, where $h$ is the Coxeter number and $r$ is the rank.
In our previous work, we observed the following pattern.

\begin{conjecture}[\cite{chin}]
    Let $(\Gamma, \Delta)$ be an admissible Dynkin biagram with $r$ vertices.
    The number of $\Gamma$-mutations in one period is the total number of roots in the root systems associated to each $\Gamma$-component, $h_{\Gamma}\cdot r$.
    Similarly, the number of $\Delta$-mutations in one period is the total number of roots in the root systems associated to each $\Delta$-component, $h_{\Delta}\cdot r$.
\end{conjecture}

The following corollary of Theorem \ref{mainThm} verifies the above conjecture for Dynkin diagrams under a certain initial tropical labeling.
\begin{corollary}
    Let $\Lambda$ be a Dynkin diagram of rank $r$ with Coxeter number $h$, viewed as the tensor product $\Lambda\otimes A_1$.
    Let $\lambda\in\mathbb{R}^r$ be an initial labeling for the associated tropical $T$-system.
    When $\lambda_i < 0$ for every $i\in [r]$, one period contains exactly $hr$ red mutations and $2r$ blue mutations. Moreover, in one half period, the blue mutations are in bijection with the negative simple roots of $\Lambda$, while the red mutations are in bijection with the positive roots of $\Lambda$.
\end{corollary}
\begin{proof}
    By Theorem \ref{fominBijection}, the generic form of a cluster variable associated to a positive root $\alpha$ is
    $$\alpha\mapsto \frac{P_{\alpha}(x)}{x^{\alpha}},$$
    where $P_{\alpha}(x)$ is a polynomial in $\mathbb{Z}_{\geq 0}[x]$ with nonzero constant term.
    In particular, the constant term of $P_{\alpha}$ is an integer $\geq 1$, which will be the dominant term in the polynomial, as all variables are $x_i \ll 1$.
    When $\alpha$ is a positive root, the tropicalization of this cluster variable (viewed as an element in the tropical $T$-system) will be positive.

    Thus, a mutation dependent only on cluster variables associated to positive roots will be red.
    On the other hand, cluster variables associated to the negative simple roots are of the form $x_i, i\in [r]$.
    As $x_i \ll 1$, the tropicalization of this cluster variable will be negative, and a mutation dependent only on cluster variables associated to negative simple roots will be blue.

    It remains to check no mutation will take into account a combination of cluster variables of both types.
    This follows from \cite{fomin2}, where it is shown that the bipartite belt achieves every cluster variable for finite type cluster algebras.
    So by Theorem \ref{mainThm}, each of the $\frac{(h+2)r}{2}$ distinct cluster variables appears exactly twice in one period.
    Moreover, the negative simple roots only appear at $t = 0$ as a cluster, and at the half-period as a cluster.
    Thus any tropical mutation will either be dependent on a sum of only cluster variables $x_i$ associated to negative simple roots or only cluster variables associated to positive roots.

    Moreover, there are exactly $2r$ $\Delta$-mutations from mutating all of the cluster variables dependent on negative simple roots at times $t = 0$ and $t = h_{\Gamma} + h_{\Delta}$.
    The remaining $hr$ mutations are $\Gamma$-mutations from mutating cluster variables dependent on positive roots at all other time steps.
\end{proof}

\bibliographystyle{alpha}
\bibliography{sources}

\end{document}